\documentclass[12pt]{amsart}
\usepackage[utf8]{inputenc}
\usepackage[english]{babel}
\usepackage{multirow}
\usepackage{anysize}
\usepackage{color}
\usepackage{multicol}

\usepackage{enumitem}
\usepackage{amsfonts}
\usepackage{amsmath}
\usepackage{amssymb}
\usepackage[arrow, matrix, curve, xdvi, dvips]{xy}
\usepackage{amsthm}

\usepackage{blkarray}
\usepackage{booktabs}
\usepackage{mathtools}

\usepackage{graphicx}
\usepackage[usenames,dvipsnames]{pstricks}
 \usepackage{pstricks-add}
 \usepackage{epsfig}

\newtheorem{teo}{Theorem}[section]

\newtheorem{remark}[teo]{Remark}

\newtheorem{proposition}[teo]{Proposition}

\title[Tanaka rigidity of graph Lie algebras]{Tanaka rigidity of graph Lie algebras}

\author{Mauricio Godoy Molina}

\address{Departamento de Matem\'atica y Estad\'istica, Universidad de La Frontera, Av. Francisco Salazar 01145 Temuco, Chile.}

\subjclass[2010]{53C24; 17B30; 17B66; 05C20; 05C25}

\keywords{2-step nilpotent Lie algebras, labeled directed graph, infinitesimal symmetries, Tanaka prolongation.}

\email{mauricio.godoy@ufrontera.cl}

\thanks{This research is partially supported by Fondecyt \#1181084}

\begin{document}
\maketitle

\begin{abstract}
We give sufficient conditions on a labeled direct graph to determine whether the Tanaka prolongation of its associated Lie algebra is infinite-dimensional. In the case that all directed edges are labeled differently, the corresponding graph Lie algebra is of infinite type if and only if the graph has a vertex of degree one.
\end{abstract}

\section{Introduction}

The main goal of the present note is to provide some easy-to-check criteria to determine whether a large family of differential systems has an infinite dimensional Lie algebra of symmetries. These differential systems are related to directed graphs and have been actively studied in recent years after being introduced by Dani and Mainkar in \cite{mainkar1}, see for example \cite{aaa,CMS,mainkar2,GL}. 

Let us first recall the notion of symmetries for differential systems. A differential system is a pair $(M,D)$, where $M$ is a differentiable manifold and $D\hookrightarrow TM$ is a distribution on $M$. These objects and their many properties are very interesting since they encode valuable information in many areas of mathematics. As a very incomplete sample of the context where differential systems appear, the interested reader can check references in sub-Riemannian geometry \cite{abb,m}, geometric control theory \cite{j}, and some specific connections between differential geometry and analysis \cite{klr}.

The modern framework to study the symmetries of a differential system $(M,D)$ is based on the study of the Lie algebra of infinitesimal symmetries
\begin{equation}\label{eq:sym}
{\rm sym}(M,D)=\{X\in{\mathfrak X}(M)\;|\;[X,\Gamma(D)]\subseteq\Gamma(D)\},
\end{equation}
where $\Gamma(D)$ denotes the vector fields on $M$ that take values in the distribution $D$. The vector fields in the algebra ${\rm sym}(M,D)$ can be seen as derivatives of diffeomorphisms of $M$ preserving $D$. Determining whether ${\rm sym}(M,D)$ is finite or infinite dimensional has important consequences. When infinite dimensional, the system might admit a continuous family of admissible deformations which can be exploited to develop analytic tools, such as in \cite{KR,klr}. When finite dimensional, the system is called rigid and an interesting problem arises concerning the structure of the algebra, such as in the classical paper by \'E. Cartan \cite{c} or a more recent generalization \cite{ben}.

One of the most important tools to study the Lie algebra \eqref{eq:sym} is the well-known Tanaka prolongation \cite{tanaka}. This lineal algebraic procedure produces a Lie algebra ${\rm sym}(N,{\mathfrak n}_{-1})$ inductively in the case that $N$ is a (connected and simply connected) nilpotent Lie group associated to a graded nilpotent Lie algebra ${\mathfrak n}={\mathfrak n}_{-\mu}\oplus\cdots\oplus{\mathfrak n}_{-1}$. Even though each step of the Tanaka prolongation is explicit, it is in general very difficult to compute. Nevertheless, some general abstract results can be found, for example it is known that in step $\mu=2$ ``most'' Lie algebras ${\rm sym}(N,{\mathfrak n}_{-1})$ are finite dimensional, see \cite{gkmv}. We call this phenomenon Tanaka rigidity.

After this necessary detour, let us focus our attention on the differential systems that are of interest for this paper. Given a finite simple directed graph $(V,E)$, one can label its edges via a surjective labeling function $c\colon E\to{\mathcal C}$, where ${\mathcal C}=\{c_1,\dotsc,c_m\}$ is the set of labels. We say that the triplet $G=(V,E,c)$ is a labeled directed graph, see \cite{GL}. Enumerating the vertices $\{x_1,\dotsc,x_n\}$, one can define the vector space
\begin{equation}
{\rm Lie}(G)={\rm span}_F\{x_1,\dotsc,x_n,c_1,\dotsc,c_m\}
\end{equation}
where $F$ is an arbitrary field. With the operation
\[
[x_i,x_j]_{{\rm Lie}(G)}=\begin{cases}c_l,&\mbox{if }c(\overrightarrow{x_ix_j})=c_l\\-c_l,&\mbox{if }c(\overrightarrow{x_jx_i})=c_l\\0,&\mbox{otherwise}.\end{cases}\quad,\quad[x_i,c_l]_{{\rm Lie}(G)}=[c_l,c_m]_{{\rm Lie}(G)}=0
\]
the vector space ${\rm Lie}(G)$ becomes a 2-step nilpotent Lie algebra with a natural grading
\[
{\rm Lie}(G)={\mathfrak g}_{-2}\oplus{\mathfrak g}_{-1},\quad{\mathfrak g}_{-2}={\rm span}_F\{c_1,\dotsc,c_m\},\quad{\mathfrak g}_{-1}={\rm span}_F\{x_1,\dotsc,x_n\},
\]
generated by ${\mathfrak g}_{-1}$. These algebras are known as graph Lie algebras.

Although this definition makes sense for arbitrary fields, we will assume for the rest of the paper that $F={\mathbb R}$, obtaining a differential system $({\rm LIE}(G),{\mathfrak G}_{-1})$, where ${\rm LIE}(G)$ is the unique (up to isomorphism) connected and simply connected Lie group with Lie algebra ${\rm Lie}(G)$ and ${\mathfrak G}_{-1}$ is the distribution generated by left-translating ${\mathfrak g}_{-1}$. 

With all these concepts at hand, we can be state more precisely the main result of the present article: we prove that if all labels of the graph $G$ are different, then the Tanaka prolongation of $({\rm LIE}(G),{\mathfrak G}_{-1})$ is infinite dimensional if and only is $G$ has a vertex of degree one. Even though we also analyze the more general situation in which some directed edges have the same label, a complete characterization of Tanaka rigidity for arbitrary labeled directed graphs remains elusive.

\section{Tanaka rigidity for graph Lie algebras}

For the rest of the present article we follow the notation from the introduction: $G=(V,E,c)$ denotes a labeled directed graph with vertices $V=\{x_1,\dotsc,x_n\}$ and surjective labeling function $c\colon E\to{\mathcal C}=\{c_1,\dotsc,c_m\}$. 

The following hypothesis will play an important role in some of the main results of this paper.
\begin{equation*}\label{hyp:H}
m=|{\mathcal C}|=|E|,\mbox{ that is, every edge is labeled uniquely}.\tag{{\bf H}}
\end{equation*}

It is interesting to observe that this assumption was implicit in \cite{mainkar1}, where graph Lie algebras were introduced. As shown in \cite{GL}, hypothesis \eqref{hyp:H} implies that changing the orientation of edges produces isomorphic Lie algebras. Relaxing this hypothesis, as suggested in \cite{CMS}, has the obvious advantage of generality but the cost is that, in general, there may appear non-trivial Abelian factors in ${\mathfrak g}_{-1}$, see \cite{mainkar2}. Under hypothesis \eqref{hyp:H} this does not occur.

\begin{proposition}\label{prop:center}
Assume $G$ is a connected graph satisfying \eqref{hyp:H}. Then 
\[
{\mathcal Z}({\rm Lie}(G))\cap{\mathfrak g}_{-1}=\{0\},
\]
where ${\mathcal Z}({\rm Lie}(G))$ denotes the center of the graph Lie algebra ${\rm Lie}(G)$.
\end{proposition}

\begin{proof}
Let $x=\sum_{i=1}^{|V|}\alpha_ix_i\in{\mathcal Z}({\rm Lie}(G))\cap{\mathfrak g}_{-1}$ and let $v\in V$ be any vertex of $G$, then
\[
[v,x]_{{\rm Lie}(G)}=\sum_{x_i\in{\mathcal N}_{v}}\alpha_i[v,x_i]_{{\rm Lie}(G)}=0
\]
is a linear combination of labels. Since $\dim{\mathfrak g}_{-2}=|{\mathcal C}|=|E|$, it follows that the set
\[
\{[v,x_i]_{{\rm Lie}(G)}\colon x_i\in{\mathcal N}_v\}
\]
is linearly independent. As a consequence, we conclude that $\alpha_i=0$, whenever $x_i\in{\mathcal N}_{v}$. As $v\in V$ was chosen arbitrarily, we have that $x=0$.
\end{proof}

To prove the main abstract result of the present paper, we will use an adapted form of a criterion found by N. Tanaka in \cite[Corollary 2 of Theorem 11.1]{tanaka}. This criterion has been used extensively in the literature and it is often called the {\em corank one condition}. It states that a graded nilpotent Lie algebra ${\mathfrak g}={\mathfrak g}_{-\mu}\oplus\cdots\oplus{\mathfrak g}_{-1}$ is of infinite type if, and only if, there exists $x\in{\mathfrak g}_{-1}\setminus\{0\}$ and a hyperplane $\Pi\subset{\mathfrak g}_{-1}$, such that
\[
[x,y]=0\quad\mbox{for all }y\in\Pi.
\]

\begin{teo}\label{th:deg1}
If there exists a vertex $v\in V$ of degee one, then $Lie(G)$ is of infinite type.
\end{teo}

\begin{proof}
Since $\deg(v)=1$, then there exists only one vertex $w\in V$ such that 
\[
[v,w]_{{\rm Lie}(G)}\neq0.
\]
Consider the hyperplane $\Pi={\rm span}\{V\setminus\{w\}\}\subset{\mathfrak g}_{-1}$. It follows directly that
\[
[v,u]_{{\rm Lie}(G)}=0\quad\mbox{for all }u\in V\setminus\{w\},
\]
since $\deg(v)=1$. The conclusion follows from the corank one condition.
\end{proof}

As is shown in \cite[Example 2.3]{GL}, it is possible to find two labeled directed graphs ${\mathcal G}_1$ and ${\mathcal G}_2$, where ${\mathcal G}_1$ does not have a vertex of degree one while ${\mathcal G}_2$ does, and such that ${\rm Lie}({\mathcal G}_1)\cong{\rm Lie}({\mathcal G}_2)$. As a consequence, the Lie algebra ${\rm Lie}({\mathcal G}_1)$ is of infinite type. The existence of the isomorphism in the cited example follows from the classification given in \cite{magnin}, nevertheless, in general, it can be quite tricky to find such isomorphism explicitly. 

The previous observation clearly implies that Theorem \ref{th:deg1} gives only a sufficient condition. The following result generalizes Theorem \ref{th:deg1}, implying that the graph Lie algebra ${\rm Lie}({\mathcal G}_1)$ above is of infinite type without resorting to the explicit isomorphism.

\begin{proposition}
Suppose there exists a vertex $v\in V$ of degree $\deg(v)=k\geq2$ such that
\begin{equation}\label{eq:eqlab}
\dim{\rm span}\{[v,w_1]_{{\rm Lie}(G)},\dotsc,[v,w_k]_{{\rm Lie}(G)}\}=1
\end{equation}
where $w_1,\dotsc,w_k\in V$ are the neighbors of $v$. Then $Lie(G)$ is of infinite type.
\end{proposition}

\begin{proof}
Condition \eqref{eq:eqlab} means that all the edges starting or ending in $v$ have the same label. Without loss of generality, by renaming or changing signs if necessary, suppose the directed edge $\overrightarrow{vw_1}$ exists and let $c\in{\mathcal C}$ be its label. Then there exist exponents $i_2,\dotsc,i_k\in\{0,1\}$ such that
\[
[v,w_\ell]_{{\rm Lie}(G)}=(-1)^{i_\ell}c\quad,\quad\ell=2,\dotsc,k.
\]
From this equalities, it follows directly that
\[
[v,w_1-(-1)^{i_2}w_2]_{{\rm Lie}(G)}=\cdots=[v,w_1-(-1)^{i_k}w_k]_{{\rm Lie}(G)}=0.
\]
The claim is a consequence of the corank one condition considering the hyperplane
\[
\Pi={\rm span}\{V\setminus{\mathcal N}_v,w_1-(-1)^{i_2}w_2,\dotsc,w_k-(-1)^{i_2}w_k\}\subset{\mathfrak g}_{-1}
\]
where ${\mathcal N}_v=\{w_1,\dotsc,w_k\}$ is the neighborhood of $v$.
\end{proof}

Note that the corank one condition is obviously equivalent to the existence of an element $x\in{\mathfrak g}_{-1}$ and a hyperplane $\Pi\subset{\mathfrak g}_{-1}$ such that
\[
\ker{\rm ad}_x\cap{\mathfrak g}_{-1}=\Pi,
\]
or equivalently
\begin{equation}\label{eq:dimcount}
\dim(\ker{\rm ad}_x\cap{\mathfrak g}_{-1})=\dim{\mathfrak g}_{-1}-1.
\end{equation}
Using this rewriting of the corank one condition, we can show that under hypothesis \eqref{hyp:H}, Theorem \ref{th:deg1} gives an equivalence.

\begin{teo}
Assuming hypothesis \eqref{hyp:H}, the following equality
\[
\dim(\ker{\rm ad}_v\cap{\mathfrak g}_{-1})=|V|-\deg(v)
\]
holds for any vertex $v\in V$. 

As a consequence, in the case that every edge of $G$ is labeled differently, then the graph Lie algebra ${\rm Lie}(G)$ is of infinite type if, and only if, $G$ has a vertex of degree one.
\end{teo}

\begin{proof}
It is enough to show that \eqref{hyp:H} implies
\[
\ker{\rm ad}_v\cap{\mathfrak g}_{-1}={\rm span}\{V\setminus{\mathcal N}_v\},
\]
for any vertex $v\in V$. 

It is easy to see that if $w\in V\setminus{\mathcal N}_v$, then ${\rm ad}_v(w)=[v,w]_{{\rm Lie}(G)}=0$, that is the inclusion
\[
\ker{\rm ad}_v\cap{\mathfrak g}_{-1}\supset{\rm span}\{V\setminus{\mathcal N}_v\}
\]
holds without assuming hypothesis \eqref{hyp:H}. On the other hand, if $V=\{x_1,\dotsc,x_n\}$
\[
x=\sum_{i=1}^n\alpha_ix_i\in\ker{\rm ad}_v\cap{\mathfrak g}_{-1}\implies[v,x]_{{\rm Lie}(G)}=\sum_{i=1}^n\alpha_i[v,x_i]_{{\rm Lie}(G)}=\sum_{x_i\in{\mathcal N}_v}\alpha_i[v,x_i]_{{\rm Lie}(G)}=0.
\]
Under hypothesis \eqref{hyp:H}, it follows that $\alpha_i=0$, whenever $x_i\in{\mathcal N}_v$, in the same way as in Proposition \ref{prop:center}. Therefore we see that
\[
x=\sum_{i=1}^n\alpha_ix_i=\sum_{x_i\in V\setminus{\mathcal N}_v}\alpha_ix_i\in{\rm span}\{V\setminus{\mathcal N}_v\}\implies\ker{\rm ad}_v\cap{\mathfrak g}_{-1}\subset{\rm span}\{V\setminus{\mathcal N}_v\}.
\]

The equality of the statement follows immediately from counting dimensions. The fact that, under hypothesis \eqref{hyp:H} the Tanaka prolongation of ${\rm Lie}(G)$ is infinite dimensional if and only if $G$ has a vertex of degree one, follows from the rewritten corank one condition shown in Equation \eqref{eq:dimcount}.
\end{proof}

Following \cite{ben}, it is possible to describe explicitly the Tanaka prolongation ${\rm Prol}({\rm Lie}(K_n))$ under hypothesis \eqref{hyp:H}, since obviously ${\rm Lie}(K_n)$ is a free nilpotent Lie algebra of step 2 
\[
{\rm Prol}({\rm Lie}(K_n))=\begin{cases}{\mathfrak j},&n=2,\\B_n,&n\geq3,\end{cases}
\]
where ${\mathfrak j}$ is the infinite dimensional prolongation of the Heisenberg Lie algebra (see \cite{KR,klr}). In the general case that $G\neq K_n$, then a complete description of the Tanaka prolongation of ${\rm Lie}(G)$ seems to be out of reach, due to the complexity of the computations. 



\begin{remark}
Since the structure constants for the basis $\{x_1,\dotsc,x_n,c_1,\dotsc,c_m\}$ of ${\rm Lie}(G)$ are 0 or $\pm1$, then Malcev's criterion asserts that ${\rm LIE}(G)$ has a lattice $\Lambda$, see \cite{malcev}. It follows that each graph Lie algebra gives rise naturally to a nilmanifold ${\rm LIE}(G)/\Lambda$. It is interesting -- at least to the author -- to study if one can read geometric or analytic information of the nilmanifold from the combinatorial data of the graph $G$. These potential interplays will be explored in future research.
\end{remark}

\section*{Acknowledgments}

The author would like to thank professors Carolina Araujo and Misha Verbitsky for some interesting mathematical discussions when visiting IMPA, Brazil. Several of the ideas presented in this note were developed during that visit. The author would also like to express his gratitude to Dr. Diego Lagos for reading earlier versions of the manuscript.


\end{document}